\documentclass[12pt]{amsart}

\let\lable\label
\def\label#1{\marginpar{#1}\lable{#1}}

\usepackage{amsmath}
\usepackage{amssymb}

\theoremstyle{plain}
\newtheorem{theorem}{Theorem}[section]
\newtheorem{lemma}[theorem]{Lemma}
\newtheorem{proposition}[theorem]{Proposition}
\newtheorem{corollary}[theorem]{Corollary}

\theoremstyle{definition}
\newtheorem{definition}{Definition}[section]

\newtheorem{innerquestion}{Question}[section]

\theoremstyle{remark}
\newtheorem{remark}{Remark}

\newbox\allthequestions

\newenvironment{question}%
{\setbox0=\vbox\bgroup\innerquestion}
{\endinnerquestion\egroup
 \global\setbox\allthequestions=\vbox{\unvbox\allthequestions
                                      \bigskip
                                      \unvcopy0}
 \unvbox0}

\newcommand\printthequestions{\section{Questions}\par
In this section we list all the questions with their original
numbering.\par
\unvbox\allthequestions}

\numberwithin{equation}{section}

\DeclareMathOperator{\fix}{fix}
\DeclareMathOperator{\fin}{fin}

\DeclareMathOperator{\cf}{cf}
\DeclareMathOperator{\dom}{dom}

\newcommand{\tie}[1]
{\,\raise-5pt\hbox{
${\buildrel{\displaystyle{\rhd}\!\!{\lhd}}\over{
\scriptstyle
#1}}$}\,}

\newcommand{\forces}[2]{\Vdash_{#1} \mbox{``} #2 \mbox{''}}

\newcommand{\Naturals}{{\mathbb N}}

\newcommand{\betan}{\beta{\mathbb N} \setminus {\mathbb N}}
\newcommand{\Nstar}{\Naturals^*}

\title{Tie-points and fixed-points in $\Naturals^*$}

\author[A. Dow]{Alan Dow}
\address{}
\author[S. Shelah]{Saharon Shelah}
\address{Department of Mathematics, Rutgers University, Hill Center,
 Piscataway, 
 New Jersey, U.S.A. 08854-8019}
\curraddr{Institute of Mathematics\\Hebrew University\\
Givat Ram, Jerusalem 91904, Israel}
\email{shelah@math.rutgers.edu}
\date{\today}
\thanks{
Research of the first author was supported by NSF grant No. NSF-.
The research of the second  author was supported by The Israel Science
Foundation founded by the Israel Academy of Sciences and Humanities, and
by NSF grant No. NSF- . This is paper number 916
 in the second  author's personal listing}
\keywords{automorphism, Stone-Cech, fixed points}
\subjclass{03A35}
\begin{document}
\maketitle
\begin{abstract}
A point $x$ is a (bow) tie-point of a space $X$ if $X\setminus \{x\}$
can be partitioned into (relatively) clopen sets each with $x$ in its
closure.  Tie-points have appeared in the construction of non-trivial
autohomeomorphisms of $\betan$ (e.g. \cite{veli.oca,ShSt735}) and in the
recent study of (precisely) 2-to-1 maps on $\betan$. In these cases
the tie-points have been the unique fixed point of an involution on
$\betan$. This paper is motivated by the search for 2-to-1 maps and
obtaining tie-points of strikingly differing characteristics. 
\end{abstract}
\bibliographystyle{plain}

\section{Introduction}

A point $x$ is a tie-point of a space $X$ if there are closed sets
$A,B$ of $X$ such that $\{x\}=A\cap B$ and $x$ is an adherent point of
each of $A$ and $B$. 
 We picture (and denote) this as  $X = A\tie{x} B$ where $A,
B$ are the closed sets which have a unique common accumulation point
$x$ and
say that $x$ is a tie-point as witnessed by $A,B$.  Let $A \equiv_x B$
mean that there is a homeomorphism from 
$A$ to $B$ with $x$ as a fixed point. If $X = A\tie{x} B$ and
$A\equiv_x B$, then there is an involution $F$ of $X$ (i.e. $F^2 = F$)
such that $\{x\} = \fix(F)$. In this case we will say that $x$ is a
symmetric tie-point of $X$. 

An autohomeomorphism $F$ of $\betan$ (or $\Nstar$)
is said to be {\em trivial\/} if
there is a bijection $f$ between cofinite subsets of $\Naturals$ such
that $F = \beta f \restriction \betan$.  If $F$ is a trivial
autohomeomorphism, then $\fix(F)$ is clopen; so of course $\betan$
will have no symmetric tie-points in this case if all
autohomeomorphisms are trivial.

If $A$ and $B$ are arbitrary compact spaces, and if $x\in A$ and $y\in
B$ are accumulation points, then let $A\tie{x{=}y} B$ denote the
quotient space of $A\oplus B$ obtained by identifying $x$ and
$y$ and let $xy$ denote the collapsed point.
 Clearly the point $xy$ is a tie-point of this space. 

We came to the study of tie-points via the following observation.

\begin{proposition} If $x,y$ are symmetric tie-points of $\betan$ as
  witnessed by $A,B$ and $A',B'$ respectively, then there is a 2-to-1
  mapping from $\betan$ onto the space $A\tie{x{=}y} B'$.
\end{proposition}

The proposition holds more generally if $x $ and $y$ are fixed points
of involutions $F,F'$ respectively. That is, replace $A$ by the
quotient space of $\betan$ obtained by collapsing  all sets
$\{z,F(z)\}$
 to single 
points and similary replace $B'$ by the quotient space induced by
$F'$. 
 It is an open problem to determine
if 2-to-1 continuous images of $\betan$ are homeomorphic to
$\betan$ \cite{RLevy}. It is known to be true if CH \cite{DTech1}
or PFA \cite{Dtwo} holds. 

There are many interesting questions that arise naturally when
considering the concept of tie-points in $\betan$. 
Given a closed set
$A\subset \betan$, let 
$\mathcal I_{A} = \{ a\subset \Naturals : a^* \subset A\}~.$
Given an ideal $\mathcal I$ of subsets of $\Naturals$, let $\mathcal
I^{\perp} = \{ b\subset \Naturals : (\forall a\in \mathcal I)~~
a\cap b=^* \emptyset \}$ and $\mathcal I^{+} = \{ d\subset \Naturals :
(\forall a\in I)~~
 d\setminus a \notin \mathcal I^{\perp} \}$. If $\mathcal J\subset
 [\Naturals]^\omega $, let $\mathcal J^{\downarrow} = \bigcup_{J\in
   \mathcal J} \mathcal P(J)$. Say that $\mathcal J\subset \mathcal I$
 is unbounded  in  $\mathcal I$ if for each $a\in \mathcal
 I$, there is 
 a $b\in \mathcal J$ such that $b\setminus a$ is infinite.

\begin{definition} If $\mathcal I$ is an ideal of subsets of
  $\Naturals$, set $\cf(\mathcal I)$ to be the cofinality of $\mathcal
  I$; $\mathfrak b(\mathcal I)$ is the minimum cardinality of an
  unbounded family in $\mathcal I$; $\delta(\mathcal I)$ is the
  minimum cardinality of a subset $\mathcal J$ of $\mathcal I$ such
  that $\mathcal J^\downarrow$ is dense in $\mathcal I$. 
\end{definition}

If $\betan = A\tie{x} B$, then $\mathcal I_B = \mathcal I_A^\perp$ and
$x$ is the unique ultrafilter on $\Naturals$ extending $\mathcal I^+_A
\cap \mathcal I_B^+$. The character of $x$ in $\betan$ is equal to the
maximum of $\cf(\mathcal I_A)$ and $\cf(\mathcal I_B)$. 

\begin{definition} Say that a tie-point $x$  has (i)
$\mathfrak b$-type;
   (ii) $\delta$-type; respectively (iii)
 $\mathfrak b\delta$-type, 
  $(\kappa,\lambda) $ if $\betan = A\tie{x} B$ and $(\kappa,\lambda)$
  equals: 
(i) $(\mathfrak b(\mathcal I_A),\mathfrak
b(\mathcal I_B))$  (ii) $(\delta(\mathcal I_A),\delta(\mathcal I_B))$;
and (iii) each of
 $(\mathfrak b(\mathcal I_A),\mathfrak b(\mathcal I_B))$ and
$(\delta(\mathcal I_A),\delta(\mathcal I_B)) $. We will adopt the
convention to put  the smaller of the pair
$(\kappa,\lambda)$ in the first coordinate.
\end{definition}

Again, it is interesting to note that if $x$ is a tie-point of
$\mathfrak b$-type $(\kappa, \lambda)$, then it
is uniquely determined (in $\betan$) by $\lambda$ many subsets of
$\Naturals$ since $x$ will be the unique point extending the family
$((\mathcal J_A)^\downarrow )^+ \cap ((\mathcal J_B)^\downarrow)^+$
where $\mathcal J_A$ and $\mathcal J_B$ are unbounded subfamilies of
$\mathcal I_A$ and $\mathcal I_B$. 

\begin{question} Can there be a tie-point in $\betan$  with $\delta$-type
  $(\kappa,\lambda)$ with $\kappa\leq \lambda$ less than the character 
of the  point? 
\end{question}

\begin{question} Can $\betan$ have tie-points of $\delta$-type 
$(\omega_1,\omega_1)$ and $(\omega_2,\omega_2)$?
\end{question}

\begin{proposition} If $\betan$ has symmetric tie-points of
  $\delta$-type $(\kappa,\kappa)$ and $(\lambda,\lambda)$, but no
  tie-points of $\delta$-type $(\kappa,\lambda)$, then $\betan$ has a
  2-to-1 image which is not homeomorphic to $\betan$. 
\end{proposition}

One could say that a tie-point $x$ was {\em radioactive\/} in $X$
 (i.e. 
\rlap{\raise5.5pt\hbox{$\bigtriangledown$}}
\hspace*{-8.2pt}\hbox{$\rhd\!\lhd$})
 if $X\setminus
\{x\}$ can be similarly split into 3 (or more) relatively 
clopen sets accumulating to $x$. This is equivalent to $X = A\tie{x}
B$ such that $x$ is a tie-point in either $A$ or $B$. 

Each point of character $\omega_1$ in $\betan$ is a radioactive
point (in particular is a tie-point). P-points of character $\omega_1$ 
are symmetric tie-points of $\mathfrak b\delta$-type
$(\omega_1,\omega_1)$, while points of character $\omega_1$ which are
not P-points will have $\mathfrak b$-type $(\omega,\omega_1)$ and
$\delta$-type $(\omega_1,\omega_1)$. If there is a tie-point of
$\mathfrak b$-type $(\kappa,\lambda)$, then of course there are
$(\kappa,\lambda )$-gaps. If there is a tie-point of $\delta$-type
$(\kappa,\lambda)$, then $\mathfrak p\leq \kappa$. 

\begin{proposition} If\label{mathp}
 $\betan = A \tie{x} B$, then $\mathfrak p\leq
\delta (\mathcal I_A)$. 
\end{proposition}

\begin{proof} If $\mathcal J\subset \mathcal I_A$ has cardinality less
  than $\mathfrak p$, there is, by Solovay's Lemma (and Bell's
  Theorem) an infinite set $C\subset \Naturals$ such that $C$ and
  $\Naturals \setminus C$ each meet
  every infinite set of the form
 $J\setminus (\bigcup \mathcal
J')$ where $\{J\}\cup\mathcal J'\in [\mathcal
 J]^{<\omega}$. 
 We may assume that $C\notin x$, hence there are $a\in \mathcal
 I_A$ and $b\in \mathcal I_B$ such that $C\subset a\cup b$. However no
 finite union from $\mathcal J$ covers $a$ showing that $\mathcal
 J^\downarrow $  can not be dense in $\mathcal I_A$. 
\end{proof}

Although it does not seem to be 
completely trivial, it can be shown that PFA
implies there are no tie-points (the hardest case to eliminate is
those of $\mathfrak b$-type $(\omega_1,\omega_1)$)). 

\begin{question} Does $\mathfrak p>\omega_1$ imply there are no
  tie-points of $\mathfrak b$-type $(\omega_1,\omega_1)$?
\end{question}

Analogous to tie-points, we also define a tie-set: say that
$K\subset\betan$ is a tie-set if $\betan = A\tie{K} B$ and $K=A\cap
B$, $A=\overline{A\setminus K}$, and $B=\overline{B\setminus K}$. Say
that $K$ is a symmetric tie-set if there is an involution $F$ such
that $K=\fix(F)$ and $F[A]=B$. 

\begin{question} If $F$ is an involution on $\betan$ such that
  $K=\fix(F)$ has empty interior, is $K$ a (symmetric) tie-set?
\end{question}

\begin{question} Is there some natural restriction on which compact
  spaces can (or can not) be homeomorphic to the fixed point set of
  some involution of $\betan$?
\end{question}

Again, we note a possible application to 2-to-1 maps.

\begin{proposition} Assume that $F$ is an involution of $\betan$ with
  $K=\fix(F)\neq\emptyset$. Further assume that $K$ has a symmetric
  tie-point $x$ (i.e. $K=A\tie{x} B$), then $\betan$ has a 2-to-1
  continuous image which has a symmetric tie-point (and possibly
$\betan$ does not have such a tie-point).
\end{proposition}

\begin{question} If $F$ is an involution of $\Nstar$, is the quotient
  space $\Nstar/F$ (in which each
 $\{x,F(x)\} $ is collapsed to a single
  point) a homeomorphic copy of $\betan$?
\end{question}

\begin{proposition}[CH] If $F$ is an involution of $\betan$, then the
  quotient space $\Nstar /F$ is homeomorphic to $\betan$. 
\end{proposition}

\begin{proof}
If $\fix(F)$ is
  empty, then $\Nstar/F$ is a 2-to-1 image of $\betan$, and so is a
  copy of $\betan$. If $\fix(F)$ is not empty, then consider two
  copies, $(\Nstar_1,F_1)$ and $(\Nstar_2,F_2)$, of $(\Nstar,F)$. The
  quotient space of $\Nstar_1/F_1 \oplus \Nstar_2/F_2$ obtained by
  identifying the two homeomorphic sets $\fix(F_1)$ and $\fix(F_2)$
  will be a 2-to-1-image of $\Nstar$, hence again a copy of
  $\Nstar$. Since $\Nstar_1\setminus \fix(F_1)$ and $\Nstar_2\setminus
  \fix(F_2)$ are disjoint and homeomorphic, it follows easily that
 $\fix(F)$ must be a P-set in $\Nstar$. It is trivial to verify that a
 regular closed set of $\Nstar$ with a P-set boundary will be (in a
 model of CH) a copy of $\Nstar$. Therefore the copy of $\Nstar_1/F_1$ in
 this final quotient space is a copy of $\Nstar$. 
\end{proof}

\section{a spectrum of tie-sets}

We adapt a method from \cite{BrSh642} to produce a model in which there
are tie-sets of specified $\mathfrak b\delta$-types.  We further
arrange that these tie-sets will themselves have tie-points but
unfortunately we are not able to make the tie-sets symmetric. In the
next section we make some progress in involving involutions.

\begin{theorem} Assume\label{main2}
 GCH and that
 $\Lambda$ is a set of regular uncountable
  cardinals such that for each $\lambda\in \Lambda$, $T_\lambda$ is a
  ${<}\lambda$-closed $\lambda^+$-Souslin tree. There is a  forcing
  extension in which there is a tie-set $K$ (of $\mathfrak
  b\delta$-type $(\mathfrak c,\mathfrak c)$) and for each $\lambda\in
  \Lambda$, there is a tie-set $K_\lambda$ of $\mathfrak b\delta$-type
  $(\lambda^+,\lambda^+)$ such that $K\cap K_\lambda$ is a single point
  which is a tie-point of $K_\lambda$. Furthermore, for
  $\mu\leq\lambda<\mathfrak c$, if $\mu\neq\lambda$ or $\lambda\notin
  \Lambda$, then there is no tie-set of $\mathfrak b\delta$-type
  $(\mu,\lambda)$. 
\end{theorem}

We will assume that our Souslin trees are well-pruned and are 
ever $\omega$-ary  branching. That is, if $T_\lambda$ is a
$\lambda^+$-Souslin tree, we assume that for each $t\in T$, $t$ has
exactly $\omega$
 immediate successors denoted $\{ t^\frown \ell :\ell\in \omega\}$
and that $\{
s\in T_\lambda : t<s\}$ has cardinality $\lambda^+$ (and so has
successors on every level). A poset is ${<}\kappa$-closed if every
directed subset of cardinality less than $\kappa$ has a lower bound. A
poset is ${<}\kappa$-distributive if the intersection of any family of
fewer than $\kappa$ dense open subsets is again dense.  For a cardinal
$\mu$, let $\mu^-$ be the minimum cardinal such that $(\mu^-)^+\geq
\mu$ (i.e. the predecessor if $\mu$ is a successor).

The main idea of the construction is nicely  illustrated by the
following.

\begin{proposition} Assume\label{tie-sets}
 that $\betan $ has no tie-sets of
  $\mathfrak b\delta$-type $(\kappa_1,\kappa_2)$ for some
  $\kappa_1\leq \kappa_2<\mathfrak c$.  Also assume that
  $\lambda^+<\mathfrak c$ is such that $\lambda^+$ is distinct from
  one of $\kappa_1,\kappa_2$ and that
$T_\lambda$ is a $\lambda^+$-Souslin
  tree and  $\{ (a_t, x_t , b_t ) : t\in T_\lambda\}\subset
  ([\Naturals]^\omega )^3 $ satisfy that, for $t<s\in T_\lambda$:
\begin{enumerate}
\item $\{a_t, x_t,b_t\}$ is a partition of $\Naturals$,
\item $x_{t^\frown j}\cap x_{t^\frown \ell}=\emptyset$ for $j<\ell$,
\item $x_s\subset^* x_t$, $a_t\subset^* a_s$, and 
$b_t\subset^* b_s$,
\item for each $\ell\in \omega$, $x_{t^\frown \ell+1}\subset^*
 a_{t^\frown
    \ell}$ and $x_{t^\frown \ell+2}\subset^* b_{t^\frown \ell}$,
\end{enumerate}
then if $\rho\in [T_\lambda]^{\lambda^+}$ is a generic branch
(i.e. $\rho(\alpha)$ is an element of the $\alpha$-th level of
$T_\lambda$ for each $\alpha\in \lambda^+$), then $K_\rho =
\bigcap_{\alpha\in \lambda^+} x_{\rho(\alpha)}^*$ is a tie-set of
$\betan$ of $\mathfrak b\delta$-type $(\lambda^+,\lambda^+)$, and 
there is no tie-set of $\mathfrak b\delta$-type
$(\kappa_1,\kappa_2)$. 

\begin{enumerate}
\setcounter{enumi}{4}
\item Assume further that $\{ (a_\xi, x_\xi, b_\xi)
 : \xi\in   \mathfrak c\}$ is a family of partitions of $\Naturals$
 such that  $\{ x_\xi : \xi \in \mathfrak c\}$ is 
a mod
  finite descending family of subsets of
$\Naturals$ such  that for each $Y\subset \Naturals$, there is a
maximal antichain  $A_Y\subset T_\lambda$ and some $\xi\in \mathfrak c$ 
such that for each $t\in A_Y$, $x_t\cap x_\xi$ is a proper subset of 
either $ Y$ or $ \Naturals\setminus Y$, 
then $K = \bigcap_{\xi\in \mathfrak c} x_\xi^*$ meets $K_\rho$ in
a single point $z_\lambda$.
\item If we assume further that for each $\xi<\eta<\mathfrak c$, 
$a_\xi \subset^* a_\eta$ and $b_\xi  \subset^* b_\eta$, and for each
$t\in T_\lambda$, $\eta$ may be chosen so that  $x_t$ meets each of 
$(a_\eta\setminus a_\xi)$ and $(b_\eta\setminus b_\xi)$, 
 then $z_\lambda$ is a   tie-point of $K_\rho$. 
\end{enumerate}
\end{proposition}

\begin{proof}
To show that $K_\rho$ is a tie-set
it is sufficient to show
that $K_\rho\subset \overline{\bigcup_{\alpha \in \lambda^+}
  a_\alpha^*}
\cap \overline{\bigcup_{\alpha \in \lambda^+}
  b_\alpha^*}$. Since $T_\lambda$ is a $\lambda^+$-Souslin tree, no 
new subset of $\lambda$ is added when forcing with $T_\lambda$. Of
course we use that $\rho$ is $T_\lambda$ is generic, 
so
assume that $Y\subset \Naturals$ and that 
some $t\in T_\lambda$ forces that $Y^*\cap K_\rho $ is not empty. We
must show that there is some $t<s$ such that $s$ forces that $a_s\cap
Y$ and $b_s\cap Y$ are both
 infinite.  However, we know that $x_{t^\frown \ell}\cap Y$ is
 infinite for each $\ell\in \omega$  since
 $t^\frown \ell\forces{T_\lambda}{K_\rho\subset x_{t^\frown\ell}^*}$.
Therefore, by condition 4,
for each $\ell\in \omega$,  $Y\cap a_{t^\frown \ell}$ and $Y\cap
b_{t^\frown \ell}$  are both infinite.

Now let $\kappa_1,\kappa_2$ be regular cardinals at least one of which
is 
distinct from $\lambda^+$. Recall that forcing with $T_\lambda$
preserves cardinals.
Assume that in $V[\rho]$, 
$K\subset \Nstar$ and $\Nstar = C\tie{K}D$ with
$\mathfrak b(\mathcal I_C)=\delta(\mathcal I_C)=\kappa_1$ and
$\mathfrak b(\mathcal I_D)=\delta(\mathcal I_D)=\kappa_2$. 
In $V$, 
let $\{ c_\gamma : \gamma\in \kappa_1\}$ be $T_\lambda$-names for the
increasing cofinal sequence in $\mathcal I_C$ and
let $\{ d_\xi : \xi \in \kappa_2\}$ be $T_\lambda$-names for the
increasing cofinal sequence in $\mathcal I_D$. Again using the fact
that $T_\lambda$ adds no new subsets of $\Naturals$ and the fact that
every dense open subset of $T_\lambda$ will contain an entire level of
$T_\lambda$, we may choose ordinals $\{ \alpha_\gamma : \gamma\in
\kappa_1\}$ 
and $\{\beta_\xi : \xi \in \kappa_2\}$ such that each $t\in
T_\lambda$, if $t$ is on level $\alpha_\gamma$ it will force a value
on $c_\gamma$ and if $t$ is on level $\beta_\xi$ it will force a value
on $d_\xi$. If $\kappa_1 < \lambda^+$, then $\sup\{\alpha_\gamma
:\gamma\in \kappa_1\}<\lambda^+$, hence there
are  $t\in T_\lambda$ which force a value on each $c_\gamma$. If
$\lambda^+<\kappa_2$, then there is some $\beta<\lambda^+$, such that
$\{\xi \in \kappa_2 : \beta_\xi \leq \beta\}$ has cardinality
$\kappa_2$. Therefore there is some $t\in T_\lambda$ such that $t$
forces a value on $d_\xi$ for a cofinal set of $\xi\in \kappa_2$. 
Of course, if neither $\kappa_1$ nor $\kappa_2$ is equal to
$\lambda^+$, then we have a condition that decided cofinal families of
each of $\mathcal I_C$ and $\mathcal I_D$. This implies that $\Nstar$
already has tie-sets of $\mathfrak b\delta$-type
$(\kappa_1,\kappa_2)$. 

If $\kappa_1<\kappa_2=\lambda^+$, then fix $t\in T_\lambda$ deciding 
$\mathfrak C = \{ c_\gamma : \gamma\in \kappa_1\}$, and let 
$\mathfrak D = \{ d\subset\Naturals : (\exists s>t)
s\forces{T_\lambda}{d^*\subset D}\}$. It follows easily that
$\mathfrak D = \mathfrak C^\perp$. But also, since forcing with
$T_\lambda$ can not raise $\mathfrak b(\mathfrak D)$ and can not lower
$\delta (\mathfrak D)$, we again have that there are tie-sets of 
$\mathfrak b\delta$-type in $V$.

The case $\kappa_1=\lambda^+ < \kappa_2$ is similar.

Now assume we have the family $\{ (a_\xi, x_\xi, b_\xi) : \xi \in
\mathfrak c\}$ as in (5) and (6)
 and set $K=\bigcap_\xi x_\xi^*$, $A=\{K\}\cup 
\bigcup\{ a_\xi^* : \xi \in \mathfrak c\}$, and $B=
\{K\}\cup \bigcup\{b_\xi^* : \xi\in \mathfrak c\}$. 
It is routine to see that (5) ensures that the family
 $\{ x_\xi \cap x_{\rho(\alpha)} : \xi\in \mathfrak c
\ \mbox{and}\ \alpha\in \lambda^+\}$ generates an ultrafilter when
 $\rho$ meets each maximal antichain $A_Y$ ($Y\subset \Naturals$).
Condition (6) clearly ensures that $A\setminus K $ and $B\setminus K$
each meet $(x_\xi\cap x_{\rho(\alpha)})^*$ for each $\xi\in \mathfrak
c$ and $\alpha\in \lambda^+$.  Thus $A\cap K_\rho$ and $B\cap K_\rho$
witness that $z_\lambda$ is a tie-point of $K_\rho$. 
\end{proof}

Let $\theta$ be a regular cardinal greater than $\lambda^+$ for all
$\lambda\in \Lambda$. We will need the following well-known Easton
lemma (see \cite[p234]{Jech}). 

\begin{lemma} Let $\mu$ be a regular cardinal\label{distrib}
 and 
 assume that $P_1$ is a poset
  satisfying the $\mu$-cc. Then any ${<}\mu$-closed poset $P_2$ 
remains   ${<}\mu$-distributive after forcing with $P_1$. Furthermore
any ${<}\mu$-distributive poset remains ${<}\mu$-distributive after forcing
 with a poset of cardinality less than $\mu$.
\end{lemma}

\begin{proof}  Recall that a poset $P$ is ${<}\mu$-distributive if
  forcing with it does not add, for any $\gamma<\mu$, 
 any new $\gamma$-sequences of
  ordinals. Since $P_2$ is ${<}\mu$-closed, forcing with $P_2$ does not
  add any new antichains to $P_1$. Therefore it follows that
forcing with $P_2$ preserves that $ P_1$ has the $\mu$-cc
and that for every $\gamma<\mu$, each $\gamma$-sequence of ordinals 
in the forcing extension by $P_2\times P_1$ is really just a
$P_1$-name. Since forcing with $P_1\times P_2$ is the same as
  $P_2\times P_1$, this shows that in the extension by $P_1$,
 there are no new $P_2$-names of $\gamma$-sequences of ordinals.

Now suppose that $P_2$ is $\mu$-distributive and that $P_1$ has
cardinality less than $\mu$. Let $\dot D$ be a $P_1$-name of a dense
open subset  of $P_2$. For each $p\in P_1$, let $D_p\subset P_2$ be the
set of all $q$ such that some extension of $p$ forces that $q\in \dot
D$. Since $p$ forces that $\dot D$ is dense and that $\dot D\subset
D_p$, it follows that $D_p$  is dense (and open). Since $P_2$ is
$\mu$-distributive, $\bigcap_{p\in P_1} D_p$ is dense and  is clearly 
going to be a subset of $\dot D$. Repeating this argument for at most
$\mu$ many $P_1$-names of dense open subsets of $P_2$ completes the
proof. 
\end{proof}

We recall the definition of Easton supported product of posets (see
\cite[p233]{Jech}).

\begin{definition} If $\Lambda$ is a set of cardinals and $\{
  P_\lambda : \lambda\in \Lambda\}$ is a set of posets, then we will
  use $\Pi_{\lambda\in \Lambda} P_\lambda$ to denote the collection of
  partial functions $p$ such that 
\begin{enumerate}
\item $\dom(p)\subset \Lambda$,
\item $|\dom(p)\cap \mu| < \mu$ for all
  regular cardinals $\mu$,
\item $p(\lambda)\in P_\lambda$ for all $\lambda\in \dom(p)$.
\end{enumerate}
This collection is a poset when ordered by $q<p$ if $\dom(q)\supset
\dom(p)$ and $q(\lambda)\leq p(\lambda)$ for all $\lambda\in
\dom(p)$. 
\end{definition}

\begin{lemma}
For\label{chain.condition}
 each cardinal $\mu$, $\Pi_{\lambda\in \Lambda\setminus \mu^+}
T_{\lambda}$  is ${<}\mu^+$-closed and, if $\mu$ is regular,
$\Pi_{\lambda\in \Lambda\cap \mu} T_\lambda$ has cardinality at most
$2^{<\mu} \leq \min(\Lambda\setminus\mu)$.
\end{lemma}

\begin{lemma} If\label{chains}
 $P$ is ccc and $G\subset P\times \Pi_{\lambda\in
    \Lambda} T_\lambda$ is generic, then in $V[G]$, for any $\mu$
    and any family $\mathcal A\subset [\Naturals]^\omega$ with
    $|\mathcal A|=\mu $:
\begin{enumerate}
\item  if $\mu\leq\omega$, then $\mathcal A$ is a member of $V[G\cap
  P]$; 
\item if $\mu=\lambda^+, \lambda\in \Lambda$, 
then there is an $\mathcal A'\subset
  \mathcal A$ of cardinality $\lambda^+ $ such that $\mathcal A'$ is a
  member of $V[G\cap (P\times T_\lambda)]$;
\item if $\mu^- \notin \Lambda$, then there is an $\mathcal A'\subset
  \mathcal A$ of cardinality $\mu $ which is a member of $V[G\cap
  P]$. 
\end{enumerate}
\end{lemma}

\begin{corollary}
If\label{omit}
 $P$ is ccc and $G\subset P\times \Pi_{\lambda\in \Lambda}
T_\lambda$ is generic, then for any $\kappa \leq\mu  < \mathfrak c$
such that either $\kappa \neq \mu $ or 
$\kappa \notin \{\lambda^+ : \lambda\in \Lambda\}$, 
if there is a tie-set of $\mathfrak b\delta$-type $(\kappa,\mu)$ in
$V[G]$, then there is such a tie-set in $V[G\cap P]$.
\end{corollary}

\begin{proof} Assume that $\betan = A \tie{K} B$ in $V[G]$ with $\mu=
  \mathfrak b(A)$  and $\lambda=\mathfrak b(B)$. Let $\mathcal
  J_A\subset\mathcal I_A$ be an increasing mod finite chain, of order type
  $\mu$, which is
  dense in $\mathcal I_A$. Similarly let $\mathcal J_B\subset\mathcal
  I_B$ be such a chain of order type $\lambda$. By Lemma \ref{chains},
  $\mathcal J_A$ and $\mathcal J_B$ are subsets of
  $[\Naturals]^\omega\cap V[G\cap P] = [\Naturals]^\omega$. 
Choose, if possible $\mu_1\in \Lambda$ such that $\mu_1^+=\mu$ and
$\lambda_1\in \Lambda$ such that $\lambda_1^+=\lambda$. 
Also by
  Lemma \ref{chains}, we can, by passing to a subcollection, assume
  that $\mathcal J_A\in V[G\cap (P\times T_{\mu_1})]$ (if there is no
  $\mu_1$, then let $T_{\mu_1}$ denote the trivial order). Similarly,
  we may assume that $\mathcal J_B\in V[G\cap (P\times
  T_{\lambda_1})]$. Fix a condition $q\in G \subset (P\times
  \Pi_{\lambda\in \Lambda} T_\lambda)$ which forces that
$(\mathcal J_A)^\downarrow$ is a $\subset$-dense subset of $\mathcal
I_A$,  that $(\mathcal J_B)^\downarrow$ is a $\subset$-dense subset
of $\mathcal I_B$, and that $(\mathcal I_A)^\perp = \mathcal I_B$. 

Working in the model $V[G\cap P]$ then, there is a
  family $\{ \dot a_\alpha : \alpha\in \mu\}$ of $T_{\mu_1}$-names for
  the members of $\mathcal J_A$; and a family $\{\dot b_\beta :
  \beta\in \lambda_1\}$ of $T_{\lambda_1}$-names for the members of
  $\mathcal J_B$. Of course if $\mu=\lambda$ and $T_{\mu_1}$ is the
  trivial order, then $\mathcal J_A$ and $\mathcal J_B$ are already in
  $V[G\cap P]$ and we have our tie-set in $V[G\cap P]$. 

Otherwise, we assume that $\mu_1<\lambda_1$. 
Set $\mathcal A$ to be the set of all $a\subset \Naturals$ such that
there is some $q(\mu_1)\leq t\in T_{\mu_1}$ and $\alpha \in \mu$ 
such that $t\forces{T_{\mu_1}}{ a=\dot a_\alpha}$. 
Similarly let $\mathcal B$ be
the set of all $b\subset \Naturals$ such that there is some
$q(\lambda_1)\leq s\in
T_{\lambda_1}$ and $\beta\in \lambda$ such that
$s\forces{T_{\lambda_1}}{ b=\dot 
b_\beta}$. It follows from the construction
that, in $V[G]$, for any $(a',b')\in \mathcal J_A\times \mathcal J_B$, 
there is an $(a,b)\in \mathcal A\times \mathcal B$ such that
$a'\subset^* a$ and $b'\subset^* b$. Therefore the ideal generated by
$\mathcal A \cup \mathcal B$ is certainly dense. It remains only to
show that $\mathcal B\subset (\mathcal A)^\perp$. 
Consider  any $(a,b)\in
\mathcal A\times \mathcal B$, and choose $(q(\mu_1),q(\lambda_1))\leq
(t,s)\in T_{\mu_1}\times T_{\lambda_1} $ such that 
$t\forces{T_{\mu_1}}{ a\in \mathcal J_A}$ and
$s\forces{T_{\lambda_1}}{b\in \mathcal J_B}$. It follows that for any
condition $\bar q\leq q$ with $\bar q\in 
(P\times \Pi_{\lambda\in  \Lambda} T_\lambda)
$,  $\bar q(\mu_1)=t$, $\bar q(\lambda_1)=s$, we
have that 
$$
\bar q\forces
{(P\times \Pi_{\lambda\in  \Lambda} T_\lambda)}{
a\in \mathcal J_A\ \mbox{and}\ 
 b\in \mathcal J_B}\ .$$  
It is routine now to check that, in $V[G\cap P]$,
 $\mathcal A$ and $\mathcal B$
generate ideals that witness that $\bigcap \{ (\Naturals \setminus
(a\cup b))^* : (a,b)\in \mathcal A\times \mathcal B\}$ is a tie-set of
$\mathfrak b\delta$-type $(\mu,\lambda)$. 
\end{proof}

Let $T$ be the rooted tree $\{\emptyset\}\cup \bigcup_{\lambda\in 
  \Lambda} T_\lambda$  and we will force an embedding of $T$ into
$\mathcal P(\Naturals)$ mod finite. In fact, we force a structure
 $\{ (a_t, x_t, b_t) : t\in T\}$ satisfying the conditions (1)-(4)
of Proposition \ref{tie-sets}.

\begin{definition} The\label{def.Q0}
 poset $Q_0$ is defined as the set of elements
  $q = (n^q, T^q, f^q)$ where $n^q\in \Naturals$, $T^q\in
  [T]^{<\omega}$, and $f^q : n^q\times T^q \rightarrow \{0,1,2\}$. The
  idea is that $x_t$ will be $\bigcup_{q\in G} \{ j\in n^q :
  f^q(j,t)=0\}$, $a_t$ will be $\bigcup_{q\in G} \{ j\in n^q :
  f^q(j,t)=1\}$ and $b_t = \Naturals\setminus (a_t\cup x_t)$. We set
  $q<p$ if $n^q\geq n^p$, $T^q\supset T^p$, $f^q\supset f^p$ and for
  $t,s\in T^p$ and $i\in [n^p,n^q)$
\begin{enumerate}
\item if $t<s$ and $f^q(i,t)\in \{1,2\}$, then $f^q(i,s)=f^q(i,t)$;
\item if $t<s$ and $f^q(i,s)=0$, then $f^q(i,t)=0$;
\item if $t\perp s$, then $f^q(i,t)+f^q(i,s)>0$.
\item if $j\in\{1,2\}$ and$\{
t^\frown \ell,t^\frown (\ell{+}j)\}\subset T^p$ 
and $f^q(i,t^\frown ({\ell+j}))=0$, then $f^q(i,t^\frown\ell)=j$.
\end{enumerate}
\end{definition}

The next lemma is very routine but we record it for reference.

\begin{lemma} The poset $Q_0$ \label{Q0} is ccc and if $G\subset Q_0$
  is generic, the family $\mathcal X_T =
\{ (a_t, x_t, b_t) : t\in T\}$ satisfies the
  conditions of Proposition \ref{tie-sets}.
\end{lemma}

We will need some other combinatorial properties of the family
$\mathcal X_T$. 

\begin{definition} For any
$\tilde T\in [T]^{<\omega}$, we define the following ($Q_0$-names). 
\begin{enumerate} 
\item for $i\in \Naturals$, $[i]_{\tilde T} = \{ j\in \Naturals : 
 (\forall t\in \tilde T)~~ i\in x_t \mbox{\ iff\ } j\in x_t \}$,
\item the collection $\fin(\tilde T)$ is the set of $ [i]_{\tilde T} $
  which are finite.
\end{enumerate}
We abuse notation and let $\fin(\tilde T)\subset n$ abbreviate
$\fin(\tilde T)\subset \mathcal P(n)$.
\end{definition}

\begin{lemma} For\label{fin}
 each $q\in Q_0$ and each $\tilde T\subset T^q$, 
$\fin(\tilde T)\subset n^q$ and for $i\geq n_q$, $[i]_{\tilde T}$ is
infinite.  
\end{lemma}

\begin{definition} A sequence $\mathcal S_W = 
\{ (a_\xi, x_\xi, b_\xi) : \xi\in W\}$
  is a tower of $T$-splitters if for $\xi<\eta\in W$ and $t\in T$:
\begin{enumerate}
\item $\{ a_\xi,x_\xi,b_\xi\}$ is a partition of $\Naturals$,
\item $a_\xi\subset^* a_\eta$, $b_\xi\subset^* b_\eta$,
\item $x_t\cap x_\xi$ is infinite.
\end{enumerate}
\end{definition}

\begin{definition} If $\mathcal S_W$ is a tower of $T$-splitters and
  $Y$ is a subset $\Naturals$, then
  the poset $Q(\mathcal S_W, Y)$ is defined as follows. 
    Let $E_Y $ be the (possibly empty) set
  of minimal  elements of $T$ such that there is some finite $H\subset
  W$ such that 
$x_t\cap Y\cap \bigcap_{\xi\in H} x_\xi$ is
  finite. Let $D_Y = E_Y^\perp = \{ t \in T : (\forall s\in E_Y)~~
  t\perp s\}$. 
A condition $q\in Q(\mathcal S_W,Y)$ is a tuple $(n^q,a^q, x^q, b^q,
T^q, H^q) $ where
\begin{enumerate}
\item $n^q\in \Naturals$ and $\{a^q, x^q, b^q\}$ is a partition of
  $n^q$,
\item $T^q\in [T]^{<\omega}$ and $H^q\in [W]^{<\omega}$,
\item $(a_\xi\setminus a_\eta)$, $(b_\xi\setminus b_\eta)$, 
 and $(x_\eta\setminus x_\xi)$ are all
 contained in $n^q$ for $\xi<\eta\in H^q$.
\end{enumerate}
We define $q<p$ to mean $n^p\leq n^q$, $T^p\subset T^q$, $H^p\subset
H^q$,
and
\begin{enumerate}
\setcounter{enumi}{3}
\item for $t\in T^p\cap D_Y$, $x_t\cap (x^q\setminus x^p)\subset Y$,
\item $x^q\setminus x^p\subset \bigcap_{\xi\in H^p} x_\xi$,
\item $a^q\setminus a^p$ is disjoint from $b_{\max(H^p)}$,
\item $b^q\setminus b^p$ is disjoint from $a_{\max(H^p)}$.
\end{enumerate}
\end{definition}

\begin{lemma} If $W\subset \gamma$, $\mathcal S_W$ is a tower of
  $T$-splitters, and if 
$G$ is $Q(\mathcal S_W,Y)$-generic, then
$\mathcal S_W \cup \{(a_\gamma, x_\gamma, b_\gamma)\}$ is also a tower
of $T$-splitters where
  $a_\gamma =\bigcup\{ a_q : q\in G\}$, $x_\gamma=\bigcup\{ x_q: q\in
  G\}$, and $b_\gamma = \bigcup\{ b_q : q\in G\}$.  In addition, for
  each $t\in  D_Y$, $x_t\cap x_\xi\subset^* Y$ (and $x_t\cap
  x_\xi\subset^* \Naturals\setminus Y$ for $t\in E_Y$).
\end{lemma}

\begin{lemma} If $W$ does not have cofinality $\omega_1$, then
  $Q(\mathcal S_W, Y)$ is $\sigma$-centered.
\end{lemma}

As usual with $(\omega_1,\omega_1)$-gaps, $Q(\mathcal S_W,Y)$ may not
(in general)
be ccc if $W$ has a cofinal $\omega_1$ sequence.

Let $0\notin C\subset \theta$ be cofinal and assume that if
$C\cap \gamma$ is cofinal in $\gamma$ and $\cf(\gamma)=\omega_1$, then
$\gamma\in C$.

\begin{definition} Fix\label{poset}
 any well-ordering $\prec$ of $H(\theta)$. We
  define a finite support iteration sequence $\{ P_\gamma, \dot
  Q_\gamma : \gamma\in \theta\}\subset H(\theta)$.  We abuse notation
  and use  $Q_0$ rather than  $\dot Q_0$ from definition \ref{def.Q0}. 
If $\gamma\notin C$, then let $\dot Q_\gamma$ be the $\prec$-least
among the list of $P_\gamma$-names of ccc posets  in
$H(\theta)\setminus \{ \dot Q_\xi : \xi\in \gamma\}$.
 If $\gamma\in C$, then let $\dot Y_\gamma$
be the
$\prec$-least $P_\gamma$-name of a subset $\Naturals$ which
is in $H(\theta)\setminus \{ \dot Y_\xi :\xi\in C\cap \gamma\}$.
Set $\dot Q_\gamma$ to be the $P_\gamma$ name of 
 $Q(\mathcal S_{C\cap \gamma}, \dot Y_\gamma)$ adding the partition
 $\{\dot a_\gamma, \dot x_\gamma, \dot b_\gamma\}$ and, where
$\mathcal S_{C\cap \gamma}$ is the $P_\gamma$-name of the
$T$-splitting tower $\{ (a_\xi,x_\xi,b_\xi) : \xi\in C\cap \gamma\}$. 

We view the members of $P_\theta$ as functions $p$ with finite domain
(or support) denoted $\dom(p)$.
\end{definition}

The main difficulty to the proof of Theorem \ref{main2} is to prove
that the iteration $P_\theta$ is ccc. Of course, since it is a finite
support iteration, this can be proven by induction at successor
ordinals. 

\begin{lemma} For\label{main2ccc}
 each $\gamma\in C$ such that $C\cap \gamma$ has
  cofinality $\omega_1$,  $P_{\gamma+1}$ is ccc.
\end{lemma}

\begin{proof}
We proceed by induction. For each $\alpha$, define $p\in P_\alpha^*$
if $p\in P_\alpha$ and there is an $n\in \Naturals$ such that
\begin{enumerate}
\item for each $\beta\in \dom(p)\cap C$, with $H^\beta=\dom(p)\cap
  C\cap \beta$, there are subsets  
$a^\beta,x^\beta,b^\beta$ of $n$ and $T^\beta \in  [T]^{<\omega}$
 such that $p\restriction \beta \forces{P_\beta}
{p(\beta)=(n,a^\beta,x^\beta,b^\beta,T^\beta,H^\beta)}$ 
\end{enumerate}
Assume that $P_\beta^*$ is dense in $P_\beta $ and let $p\in
P_{\beta+1}$. To show that $P_{\beta+1}^*$ is dense in $P_{\beta+1}$
we must find some $p^* \leq p$ in $P_{\beta+1}^*$. If $\beta\notin C$
and $p^* \in P_{\beta}^*$ is below $p\restriction \beta$, then
$p^*\cup \{(\beta,p(\beta)\}$ is the desired element of
$P_{\beta+1}^*$. Now assume that $\beta\in C$ and assume that
$p\restriction \beta\in P_\beta^*$ and that $p\restriction \beta$
forces that $p(\beta)$ is the tuple $(n_0,a,x,b,\tilde T,\tilde
H)$. By an easy density argument, we may assume that $\tilde H\subset
\dom(p)$. Let $n^*$ be the integer witnessing that $p\restriction \beta\in
P_{\beta}^*$. Let $\zeta$ be the maximum element of $\dom(p)\cap
C\cap\beta$ and let $p\restriction\zeta \forces{P_\zeta} {p(\zeta) = 
(n^*,a^\zeta,x^\zeta,b^\zeta,T^\zeta, H^\zeta)}$ as per the definition
of $P_{\zeta+1}^*$. 
Notice that since $\tilde H\subset H^\zeta$ we have that
$$p\restriction \beta \forces{P_\beta} { 
(n^*,~a^*, x, b^*,
T^\zeta\cup \tilde T, H^\zeta\cup\{\zeta\}) \leq p(\beta)}$$
where $a^*= a\cup ([n_0,n^*)\setminus b^\zeta)$ and
$b^*= b\cup ([n_0,n^*)\cap b^\zeta)$. Defining $p^*\in P_{\beta+1}$ by
$p^*\restriction \beta = p\restriction \beta$ and $p^*(\beta) =
(n^*,~a^*, x, b^*,
T^\zeta\cup \tilde T, H^\zeta\cup\{\zeta\})$ completes the 
proof that
$P_{\beta+1}^*$ is dense in $P_{\beta+1}$, and by induction, that this
holds for $\beta=\gamma$. 

Now assume that $\{p_\alpha : \alpha\in \omega_1\}\subset
P_{\gamma+1}^*$. By passing to a subcollection, we may assume that 
\begin{enumerate}
\item the collection $\{ T^{p_\alpha(\gamma)} : \alpha \in \omega_1\}$
  forms a $\Delta$-system with root $T^*$;
\item the collection $\{\dom(p_\alpha) : \alpha\in \omega_1\}$ also
  forms a $\Delta$-system with root $R$;
\item there is a tuple $(n^*,a^*,x^*,b^*)$ so that for all $\alpha\in
  \omega_1$, $a^{p_\alpha(\gamma)} =a^*$, $x^{p_\alpha(\gamma)}=x^*$,
  and $b^{p_\alpha(\gamma)}=b^*$. 
\end{enumerate}

Since $C\cap\gamma$ has a cofinal sequence of order type $\omega_1$,
there is a $\delta\in \gamma$ such that $R\subset \delta$ and, we may
assume, $(\dom(p_\alpha)\setminus \delta) \subset
\min(\dom(p_\beta)\setminus \delta)$ for $\alpha<\beta<\omega_1$. 
Since $P_\delta$ is ccc, there is a pair $\alpha<\beta<\omega_1$ such
that $p_\alpha\restriction \delta$ is compatible with
$p_\beta\restriction \delta$. Define $q\in P_{\gamma+1}$ by
\begin{enumerate}
\item $q\restriction \delta$ is any element of $P_\delta$ which is
  below each of $p_\alpha\restriction\delta$ and
  $p_\beta\restriction\delta$,
\item if $\delta\leq \xi\in \gamma\cap\dom(p_\alpha)$, then
  $q(\xi)=p_\alpha(\xi) $,
\item if $\delta\leq \xi\in \dom(p_\beta)\setminus C$, then
  $q(\xi)=p_\beta(\xi) $,
\item if $\delta\leq \xi\in \dom(p_\beta)\cap C$, then 
$$q(\xi) = 
(n^*,a^{p_\beta(\xi)},x^{p_\beta(\xi)},b^{p_\beta(\xi)},
T^{p_\beta(\xi)}, H^{p_\beta(\xi)}\cup H^{p_\alpha(\gamma)}).$$
\end{enumerate}

The main non-trivial fact about $q$ is that it is in $P_{\gamma+1}$
which depends on the fact that, by induction on $\eta\in C\cap \gamma$, 
$q\restriction
\eta$  forces that 
$$(a_\eta\setminus a_\xi)\cup (b_\eta\setminus b_\xi)\cup
(x_\xi\setminus x_\eta) 
\subset n^* \ \mbox{for}\ \xi\in
C\cap \eta . $$ 
It now follows trivially that $q$ is below each of $p_\alpha$ and
$p_\beta$. 
\end{proof}

\bgroup
\def\proofname{Proof of Theorem \ref{main2}}

\begin{proof}
This completes the construction of the ccc poset $P$ ($P_\theta$ as
above). Let $G\subset (P\times \Pi_{\lambda\in \Lambda} T_\lambda)$ be
generic. It follows that $V[G\cap P]$ is a model of Martin's Axiom and
$\mathfrak c=\theta$.  Furthermore by applying Lemma
\ref{chain.condition} with $\mu=\omega$ and Lemma \ref{distrib},
 we have that $P_2 = 
 \Pi_{\lambda\in \Lambda} T_\lambda$ is $\omega_1$-distributive in the
 model $V[G\cap P]$. Therefore all subsets of $\Naturals$ in the model
 $V[G]$ are also in the model $V[G\cap P]$. 

Fix any  $\lambda\in
 \Lambda$ and  let $\rho_\lambda$ denote the generic branch in $T_\lambda$
 given by $G$. Let $G^\lambda$ denote the generic filter on 
$P\times \Pi\{ T_\mu : \lambda\neq \mu\in \Lambda\}$ and work in the
 model $V[G^\lambda]$. It follows easily by Lemma
 \ref{chain.condition} and Lemma \ref{distrib}, that $T_\lambda$ is a 
$\lambda^+$-Souslin tree in this model. Therefore by 
 Proposition \ref{tie-sets}, $K_\lambda =
 \bigcap_{\alpha<\lambda^+} x_{\rho_\lambda(\alpha)}^*$ is a tie-set
 of $\mathfrak b\delta$-type $(\lambda^+,\lambda^+)$ in $V[G]$.
 By the
 definition of the iteration in $P$, it follows that condition (4)
of Lemma \ref{tie-sets} is also satisfied, hence the tie-set
 $K = \bigcap_{\xi\in C} x_\xi^*$ meets $K_\lambda$ in a single
 point $z_\lambda$. A simple genericity argument
confirms  that conditions (5) and (6) of
Proposition   \ref{tie-sets} also holds, hence $z_\lambda$ is a
tie-point of  $K_\lambda$. 

It follows from Corollary \ref{omit} that there are no {\em unwanted}
tie-sets in $\betan$ in $V[G]$, at least if there are none in $V[G\cap
  P]$. Since $\mathfrak p=\mathfrak c$ in $V[G\cap P]$, it follows
from Proposition \ref{mathp} that indeed there are no such tie-sets in
$V[G\cap P]$. 
\end{proof}

\egroup

Unfortunately the next result shows that the construction 
does not provide us with our desired variety of tie-points (even with
variations in the definition of the iteration).
We do not know if  $\mathfrak b\delta$-type  
can be improved to $\delta$-type (or simply exclude tie-points
altogether). 

\begin{proposition} In the model constructed in Theorem \ref{main2},
there are no
tie-points with $\mathfrak b\delta$-type $(\kappa_1,\kappa_2)$
for any $\kappa_1\leq\kappa_2<\mathfrak c$, 
\end{proposition}

\begin{proof}
Assume that $\betan = A\tie{x}B$ and that $\delta(\mathcal
I_A)=\kappa_1$ and $\delta(\mathcal I_B)=\kappa_2$. 
It follows from Corollary \ref{omit} that we can assume that
$\kappa_1=\kappa_2 =\lambda^+$ for some $\lambda\in \Lambda$. 
Also, following the
proof of Corollary \ref{omit}, there are
$P\times T_{\lambda}$-names 
$\mathcal J_A =
\{ \tilde a_\alpha : \alpha\in \lambda^+\}$ and $P\times
T_{\lambda^+}$-names $\mathcal J_B =
\{\tilde b_\beta :
\beta\in \lambda^+\}$ such that the valuation of these names by $G$
result in  increasing (mod finite) chains in $\mathcal I_A$ and
$\mathcal I_B$ respectively whose downward closures are dense. 
Passing to $V[G\cap P]$,
since $T_{\lambda}$ has the $\theta$-cc,
there is a Boolean 
subalgebra  $\mathcal B \in [\mathcal P(\Naturals)]^{<\theta}$
 such that
each $\tilde a_\alpha$ and $\tilde b_\beta$ is a name of a member of
$\mathcal B$. Furthermore, there is an infinite $C\subset \Naturals$
such that $C\notin x$ and each of 
$b\cap C $ and $b\setminus C$ are infinite for all $b\in \mathcal B$. 
Since $C\notin x$, there is a $Y\subset \Naturals$ (in $V[G]$) such
that $C\cap Y\in \mathcal I_A$ and $C\setminus Y\in \mathcal I_B$. 
Now choose $t_0\in T_{\lambda}$
which forces this about $C$ and $Y$.
Back in $V[G\cap P]$, set
$$\mathcal A = \{ b\in \mathcal B : (\exists t_1 \leq
t_0~) ~~ t_1\forces{T_{\lambda}}
 {b \in \mathcal J_A\cup \mathcal J_B}\}~.
$$
Since $V[G\cap P]$ satisfies $\mathfrak p=\theta$
 and $\mathcal A^\downarrow $ is forced by $t_0$
 to be dense in $[\Naturals]^\omega$,
 there must be a
finite subset $\mathcal A'$ of $\mathcal A$ which covers $C$.
It also follows easily then that there must be some $a,b\in \mathcal
A'$ and 
$t_1,t_2$ each below $t_0$ 
such that $t_1\forces{T_{\lambda^+ }}{a\in \mathcal J_A}$,
 $t_2\forces{T_{\lambda^+}}{b\in \mathcal J_B}$, and $a\cap b$ is
 infinite. The final contradiction is that we will now have that $t_0$
 fails to force that $C\cap a\subset^*Y$ and $C\cap b\subset^*
 (\Naturals \setminus Y)$. 
\end{proof}

\section{$T$-involutions}

In this section we strengthen the result in Theorem \ref{main2} by
making each  $K\cap K_\lambda$ a symmetric tie-point
in $K_\lambda$ (at the expense of
weakening Martin's Axiom in 
$V[G\cap P]$).
 This is progress in producing involutions with some
control over the fixed point set but we are still not able 
 to make $K$ the fixed point set of an involution. 
A poset is said to be $\sigma$-linked if there is a countable
collection of linked (elements are pairwise compatible) which union to
the poset. The statement $\hbox{MA}({\sigma}-\hbox{linked})$
 is, of course, the
assertion that Martin's Axiom holds when restricted to $\sigma$-linked
posets.

Our approach is to replace $T$-splitting towers by the following
notion. If $f$ is a (partial) involution on $\Naturals$, let $\min(f)
= \{ n\in \Naturals : n<f(n)\}$ and $\max(f)=\{ n\in \Naturals :
f(n)<n\}$ (hence $\dom(f)$ is partitioned into $\min(f)\cup
\fix(f)\cup \max(f)$).

\begin{definition} A sequence $\mathfrak T = 
\{ (A_\xi, f_\xi) : \xi\in W\}$ is
  a tower of $T$-involutions  if $W$ is a set of ordinals and
for $\xi<\nu\in W$ and $t\in T$
\begin{enumerate} 
\item $A_\nu\subset^* A_\xi$;
\item $f_\xi^2=f_\xi$ and $f_\xi\restriction (\Naturals \setminus
  \fix(f_\xi) ) \subset^* f_\eta$;
\item $f_\xi[x_t] =^* x_t$ and $\fix(f_\xi)\cap x_t$ is infinite;
\item $f_\xi([n,m)) = [n,m)$ for $n<m$ both in $A_\xi$.
\end{enumerate}
Say that $\mathfrak T$, a tower of $T$-involutions,
is  {\bf full\/}
 if $K= K_{\mathfrak T} = \bigcap \{
 \fix(f_\xi)^* : \xi \in W\}$ is a tie-set with $\betan = A\tie{K} B$
 where $A = K\cup \bigcup \{\min(f_\xi)^* : \xi\in W\}$ and 
$B = K\cup \bigcup \{\max(f_\xi)^* : \xi\in W\}$.  
\end{definition}

If $\mathfrak T$ is a tower of $T$-involutions, then there
is a natural involution $F_{\mathcal T}$ on $\bigcup_{\xi\in W}
(\Naturals\setminus 
\fix(f_\xi))^*$, but this $F_{\mathcal T}$ need not extend to an
involution on the closure of the union - even if the tower is full.

In this section we prove the following theorem.

\begin{theorem} Assume\label{main3}
 GCH and that
 $\Lambda$ is a set of regular uncountable
  cardinals such that for each $\lambda\in \Lambda$, $T_\lambda$ is a
  ${<}\lambda$-closed $\lambda^+$-Souslin tree. Let $T$ denote the tree
  sum of $\{ T_\lambda : \lambda\in \Lambda\}$. 
There is
  forcing   extension in which there is $\mathfrak T$, a full tower of
  $T$-involutions, such that the associated tie-set $K$ has 
$\mathfrak b\delta$-type $(\mathfrak c,\mathfrak c)$  and such that
for each $\lambda\in \Lambda$, there is a
 tie-set $K_\lambda$ of $\mathfrak b\delta$-type
  $(\lambda^+,\lambda^+)$ such that 
  $F_{\mathfrak T}$ does induce an involution on $K_\lambda$ with
  a singleton fixed point set $\{z_\lambda\}=K\cap K_\lambda$.
 Furthermore, for
  $\mu\leq\lambda<\mathfrak c$, if $\mu\neq\lambda$ or $\lambda\notin
  \Lambda$, then there is no tie-set of $\mathfrak b\delta$-type
  $(\mu,\lambda)$. 
\end{theorem}

\begin{question} Can the tower $\mathfrak T$ in Theorem \ref{main3} be
  constructed so that $F_{\mathfrak T}$ extends to an involution of
  $\betan$ with $\fix(F)=K_{\mathfrak T}$?
\end{question}

We introduce $T$-tower extending forcing.

\begin{definition} If\label{tower}
 $\mathfrak T=
\{ (A_\xi, f_\xi) : \xi\in W\}$ is a tower of $T$-involutions and $Y$
is a subset of $\Naturals$, we define the poset $Q=Q(\mathfrak T,Y)$ as
follows. 
    Let $E_Y $ be the (possibly empty) set
  of minimal  elements of $T$ such that there is some finite $H\subset
  W$ such that 
$x_t\cap Y\cap \bigcap_{\xi\in H} \fix(f_\xi)$ is
  finite. Let $D_Y = E_Y^\perp = \{ t \in T : (\forall s\in E_Y)~~
  t\perp s\}$. 
A tuple $q\in Q$ if $q=(a^q, f^q, T^q, H^q)$ where:
\begin{enumerate}
\item $H^q\in [W]^{<\omega}$, $T^q\in [T]^{<\omega}$, 
and $n^q=\max(a^q) \in A_{\alpha^q}$ where $\alpha^q=\max(H^q)$,
\item $f^q$ is an involution on $n^q$,
\item $(A_{\alpha^q}\setminus n^q)
\subset A_\xi$ for each $\xi\in H^q$,
\item $\fin(T^q)\subset n^q$,
\item $f_\xi\restriction (\Naturals\setminus(\fix(f_\xi)\cup n^q))\subset
  f_{\alpha^q}$ for $\xi\in H^q$,
\item $f_{\alpha^q}[x_t\setminus n^q] = x_t\setminus n^q$ for $t\in
  T^q$,
\end{enumerate}
We define $p<q$ if $n^p\leq n^q$, and for $t\in T^p$ and $i\in 
[n^p,n^q)$:
\begin{enumerate}
\setcounter{enumi}{6}
\item $a^p=a^q\cap n^p$, $T^p\subset T^q$, and $H^p\subset H^q$,
\item $a^q\setminus a^p\subset A_{\alpha^p}$,
\item $f_{\alpha^p}(i)\neq i$ implies $f^q(i)=f_{\alpha^p}(i)$,
\item $f^q([n,m))=[n,m)$ for $n<m$ both in $a^q\setminus a^p$,
\item $f^q(x_t\cap [n^p,n^q)) = x_t\cap [n^p,n^q)$,
\item if $t\in D^p$ and $i\in x_t\cap \fix(f^q)$, then $i\in Y$
\end{enumerate}
\end{definition}

It should be clear that the involution $f$ introduced by $Q(\mathfrak
T,Y)$ satisfies that for each $t\in D_Y$, $\fix(f)\cap x_t \subset^*
Y$, and, with the help of the following density argument, that
 $\mathfrak T\cup \{(\gamma,A,f)\}$ is again a tower of
 $T$-involutions where $A$ is the infinite set introduced
 by the first
 coordinates of the conditions in the generic filter.

\begin{lemma} If\label{density}
 $W\subset \gamma$, $Y\subset \Naturals$, and 
$\mathfrak T=
\{ (A_\xi, f_\xi) : \xi\in W\}$ is a tower of $T$-involutions 
and $p\in Q(\mathfrak T,Y)$, then for any $\tilde T\in [T]^{<\omega}$,
 $\zeta\in W$, 
and any $m\in \Naturals$, 
there is a $q<p$ such that $n^q\geq m$, $\zeta\in H^q$,
$T^q\supset \tilde T$,
and $\fix(f^q)\cap (x_t\setminus n^p)$ is not empty for each $t\in
T^p$.
\end{lemma}

\begin{proof}
Let $\beta$ denote the maximum $\alpha^p$ and $\zeta$ and
 let $\eta$ denote the minimum.
Choose any $n^q\in A_{\alpha^q}\setminus m$ large enough so that
\begin{enumerate}
\item 
$f_{\alpha^p}[x_t\setminus n^q] = x_t\setminus n^q$ for $t\in \tilde
T$,
\item $f_\eta\restriction (\Naturals\setminus (n^q\cup \fix(f_\eta)))
            \subset f_{\beta}$,
\item $A_\beta\setminus A_\eta$ is contained in $n^q$,
\item $n^q\cap  [i]_{T^p}\cap  \fix(f_{\alpha^p})$
 is non-empty for each $i\in \Naturals$
such that $[i]_{T^p}$ is in the finite set
 $\{ [i]_{T^p} : i\in \Naturals\}\setminus \fin(T^p)$,
\item if $i\in x_t \cap n^q\setminus n^p$ for some $t\in D_Y\cap T^p$, 
then $Y$ meets $[i]_{T^p} \cap n^q\setminus n^p$ in at least two points.
\end{enumerate}
Naturally we also set $H^q=H^p\cup\{\zeta\}$ and
$T^q = T^p\cup\tilde T$. The choice of  $n^q$
is large enough to satisfy (3), (4), (5) and (6) of Definition \ref{tower}.
We will set $a^q = a^p\cup \{n^q\}$ ensuring (1) of Definition
\ref{tower}. Therefore for any $f^q\supset f^p$ which is an involution
on $n^q$, we will have that $q=(a^q,f^q, T^q, H^q)$ is in the
poset. We have to choose $f^q$ more carefully to ensure that $q\leq
p$. Let $S=[n^p,n^q)\cap \fix(f_{\alpha^p})$, and
$S'=[n^p,n^q)\setminus S$.
We choose $\bar f$ an involution on $S$ and set $f^q =
f^p\cup (f_{\alpha^p}\restriction S') \cup 
\bar f$. We leave it to the reader to check that it suffices
to ensure that 
 $\bar f$ sends $[i]_{T^p}\cap S $ to itself for
 each $t\in T^p$ and that $\fix(\bar f)\cap x_t\subset Y$ for 
each $t\in T^p\cap D_Y$. Since the members of
$\{  [i]_{T^p}\cap S : i\in \Naturals\}$ are pairwise disjoint we can 
define $\bar f$ on each separately.

 For each $[i]_{T^p}\cap S$ which has
even cardinality, choose two points $y_i,z_i$ from it
 so that 
if there is a
$p\in D_Y\cap T^p$ such that $[i]_{T^p}\subset x_t$, 
then $\{y_i,z_i\}\subset Y$.
Let $\bar f$ be any  involution on
$[i]_{T^p}\cap S$ so that $y_i,z_i$ are the only fixed points. 
 If $[i]_{T^p}\cap S$ has odd cardinality then choose
a point $y_i$ from it so that if
$[i]^{T^p}$   is contained in
$x_t$ for some $t\in D_y\cap T^p$, then  $y_i\in 
Y\cap [i]_{T^p}\cap S$. Set $\bar f(y_i)=y_i$ and choose 
$\bar f$  to be any fixed-point free involution on 
 $[i]^{T^p}\cap S\setminus \{y_i\}$. 
\end{proof}

Let $P_\theta$ now be the finite support iteration defined as in 
Definition \ref{poset} except for two important changes.
For $\gamma\in C$,   we replace 
 $T$-splitting towers by the obvious  inductive definition of
towers of $T$-involutions when we replace
the posets 
$\dot Q(\mathcal S_{C\cap \gamma},\dot Y_\gamma)$ by
$\dot Q(\mathfrak T_{C\cap \gamma},\dot Y_\gamma)$. For $\gamma\notin
C$ we require that $\forces{P_\gamma} {\dot Q_\gamma\ \hbox{ is}\ 
\sigma\hbox{-linked.}}$

Special (parity) properties of the family $\{x_t : t\in T\}$ are
needed to ensure that $\forces{P_\gamma}{\dot Q(\mathcal
  S_{C\cap\gamma}, \dot Y_\gamma) \hbox{\ is ccc\ }}$ even for cases
when $\cf(\gamma)$ is not $\omega_1$. 

The proof of Theorem \ref{main3} is virtually the same as the proof of
Theorem \ref{main2} (so we skip) once we have established that the
iteration is ccc.

\begin{lemma} For\label{main3ccc} each $\gamma\in C$,
$P_{\gamma+1}$ is ccc.
\end{lemma}

\begin{proof}
We again define $P_\alpha^*$ to be those $p\in P_\alpha$ for which
there is an $n\in \Naturals$ such that for each $\beta\in \dom(p)\cap
C$, there are $n\in a^\beta\subset n{+}1$, $f^\beta \in n^n$,
$T^\beta\in 
[T]^{<\omega}$, and $H^\beta = \dom(p)\cap C\cap \beta$ 
such that 
$p\restriction \beta\forces{P_\beta} {p(\beta)=(a^\beta,f^\beta,
  T^\beta, H^\beta)}$. However, in this proof we must also make some
special assumptions in coordinates other than those in $C$. For each
$\xi\in \gamma\setminus C$, we fix a collection $\{ \dot Q(\xi,n) :
n\in \omega\}$ of $P_\xi$-names so that 
$$
1\forces{P_\xi}{ \dot Q_\xi = \bigcup_n \dot Q(\xi,n)
\ \hbox{and} \
(\forall n)~~\dot Q(\xi,n) \
\hbox{is linked.}}
$$
The final restriction on $p\in P_\alpha^*$ is that for each $\xi\in
\alpha\setminus C$, there is a $k_\xi\in \omega$ 
such that $p\restriction \xi
\forces{P_\xi}{p(\xi)\in \dot Q(\xi,k_\xi)}$.

Just as in Lemma \ref{main2ccc}, Lemma \ref{density} can be used to
show by induction that $P_\alpha^*$ is a dense subset of $P_\alpha$. 
This time though, we also demand that $\dom(f^{p(0)}) = n\times
T^{p(0)}$  
is such that $T^\beta\subset T^{p(0)}$ for all $\beta \in \dom(p)\cap
C$ and some extra argument is needed because of needing to decide
values in the name $\dot Y_\gamma$ as in the proof of Lemma
\ref{density}. Let $p\in P_{\beta+1}$ and assume that $P_{\beta}^*$ is
dense in $P_\beta$. By density, we 
may assume that $p\restriction \beta\in
P_{\beta}^*$, $H^{p(\beta)}\subset \dom(p)$, $T^{p(\beta)}\subset
T^{p(0)}$, 
 and that 
$p\restriction \beta$ has decided the members of the set
$D_{\dot Y_\beta}\cap T^{p(\beta)}$.  We can assume further
 that for each $t\in 
D_{\dot Y_\beta}\cap T^{p(\beta)}$, $p\restriction\beta$ has forced a
value $ y_t\in \dot Y_\beta \cap x_t\setminus \bigcup\{ x_s : s\in
T^p\ \mbox{and}\ s\not\leq t \}$ such that $y_t > n^{p(\beta)}$. 
We are using that 
 $T$ is not finitely  branching to deduce that if
$t\in D_{\dot Y_\beta} $, then $p\restriction \beta\forces{P_\beta}{
\dot Y_\beta \cap x_t\setminus \bigcup\{ x_s : s\in
T^p\ \mbox{and}\ s\not\leq t \}\ \hbox{ is non-empty} }$
(which follows since 
$\dot Y_\beta$ must meet $x_s$ for each immediate successor $s$ of
$t$). Choose any $m$ larger than $y_t$ for each $t\in
T^{p(\beta)}$. Without loss of generality, we may assume that the
integer $n^*$ witnessing that $p\restriction \beta\in P_\beta^*$ is
at least as large as $m$ and that $n^*\in \bigcap_{\xi\in
  H^{p(\beta)}} A_{\xi}$. Construct $\bar f$ just as in Lemma
\ref{density}, except that this time there is no requirement to
actually have fixed points so one member of $\dot Y_\beta$ in each 
appropriate $[i]_{T^{p(\beta)}}$ is all that is required. Let
$\zeta=\max(\dom(p)\cap \beta)$. 
No new
forcing decisions are required of $p\restriction\beta$ in order to
construct a suitable $\bar f$, hence this shows that $p\restriction
\beta \cup \{ (\beta,q)\}$ (where $q$ is constructed below $p(\beta)$
as in Corollary \ref{density} in which $H^{p(\zeta)}\cup\{\zeta\}$ is
add to $H^q$) is the desired extension of $p$ which
is a member of $P_{\beta+1}^*$. 

Now to show that $P_{\gamma+1}$ is ccc, let $\{ p_\alpha : \alpha \in
\omega_1\}\subset P_{\gamma+1}^*$. Clearly we may assume that the
family  $\{ p_{\alpha}(0) : \alpha \in \omega_1\}$ are pairwise
compatible and that there is a single integer $n$ such that, for each
$\alpha\in \omega_1$, 
$\dom(p_{\alpha}(0)) = n\times T^\alpha$ for some $T^\alpha\in
[T]^{<\omega}$. 
Also,
we may assume that there is some
 $(a,h)$ such that, for each $\alpha$, 
$$ p_\alpha \restriction\gamma\forces{P_\gamma}{p(\gamma) =
(a,h,T^\alpha,H^\alpha)}$$
where $H^\alpha = \dom(p_\alpha)\cap C\cap \gamma$. 

The family $\{
\dom(p_\alpha) \cap \gamma : \alpha\in \omega_1\}$ may be assumed
to form a
$\Delta$-system with root $R$. For each $\xi\in R$, we may
assume that, if $\xi\notin C$, there is a single $k_\xi\in \omega$
such that, for all $\alpha$, $p\sb\alpha\restriction\xi
\forces{P_\xi}{p_\alpha(\xi) \in \dot Q(\xi,k_\xi)}$, and if $\xi\in
C$,
 then there is a single $(a_\xi,h_\xi)$ such that
 $p\sb\alpha\restriction \xi\forces{P_\xi}{
p_\alpha(\xi) = (a_\xi,h_\xi, T^\alpha, H^\alpha \cap \xi)}$.
For convenience, for each  $\xi\notin C$
let $\dot r_\xi$ be a $P_\xi$-name of a function from $\omega\times 
\dot Q_\xi^2$ such that, for each $k\in \omega$,
$$1\forces{P_\xi}{ \dot r_\xi(k,q,q')\leq q,q' ~~(\forall q,q'\in \dot
  Q(\xi,k)) }. $$

Fix any $\alpha<\beta<\omega_1$ and let $H=H^\alpha\cup H^\beta$. 
Recall that $p_\alpha(0)$ and $p_\beta(0)$ are compatible. 
Recursively define a $P_\xi$-name
$q(\xi)$ for $\xi\in \dom(p_\alpha)\cup\dom(p_\beta)$ so that 
$q\restriction\xi\Vdash_{P_\xi}{}$
$$
\hbox{``}
q(\xi) = \begin{cases}
 (n,T^\alpha\cup T^\beta,f^{p_\alpha(0)}\cup f^{p_\beta(0)}) & \xi=0\\ 
 \dot r_\xi(k_\xi,p_\alpha(\xi),p_\beta(\xi))  & \xi\in R\setminus C\\
 p_\alpha(\xi) & \xi\in \dom(p_\alpha)\setminus (R\cup C)\\
 p_\beta (\xi) & \xi\in \dom(p_\beta )\setminus (R\cup C)\\
 (a_\xi,h_\xi, T^\alpha\cup T^\beta, H\cap \xi) & \xi\in C.
\end{cases}
\hbox{''.}
$$
Now we check that $q\in P_{\xi}$ by induction on
$\xi\in \gamma+1$. 

The first thing to note is that not only is this
true for $\xi=1$, but also that $q(0)\forces{Q_0} {\fin(T^\alpha\cup
T^\beta)\subset n}$. Since $p_\alpha$ and $p_\beta$ are each in
$P_{\gamma+1}^*$, this show that condition (4) of Definition
\ref{tower} will hold in all coordinates in $C$.

We also prove, by induction on $\xi$, that $q\restriction \xi$ forces
 that for $\eta< \delta  $ both in $H\cap \xi$ 
and $t\in T^\alpha\cup T^\beta$,
$f_\delta [x_t\setminus n] = x_t\setminus n$, 
 $f_\eta\restriction
(\Naturals\setminus (\fix(f_\eta)\cup n))\subset f_\delta $
 and $A_\delta \setminus n\subset A_\eta$. 

Given $\xi\in H$ and the assumption that 
$q\restriction\xi\in P_\xi$,
and $\alpha=\alpha^{q(\xi)} = \max(H\cap \xi)$, condition (3), (5),
and (6) 
of Definition \ref{tower} hold by the inductive hypothesis above. 
It follows then that $q\restriction \xi\forces{P_\xi} {q(\xi)\in \dot
Q_\xi}$. 
By the definition of the ordering on $\dot Q_\xi$, given that $H\cap
\xi = H^{q(\xi)}$ and $T^\alpha\cup T^\beta = T^{q(\xi)}$, it follows
that the inductive hypothesis then holds for $\xi+1$. 
 
It is trivial for $\xi\in \dom(q)\setminus C$, that 
$q\restriction \xi\in P_\xi$ 
implies that $q\restriction\xi\forces{P_\xi}{q(\xi)\in \dot Q_\xi}$. 
This completes the proof that $q\in P_{\gamma+1}$, and it is trivial
that $q$ is below each of $p_\alpha$ and $p_\beta$.
\end{proof}

\begin{remark} If we add a trivial tree $T_1$ to the collection 
$\{ T_\lambda : \lambda\in \Lambda\}$ (i.e. $T_1$ has only a root),
then the root of $T$ has a single extension which is a maximal node
$t$, 
and with no change to the proof of Theorem \ref{main3}, one obtains
that $F$ induces an automorphism on $x_t^*$ with a single fixed point.
Therefore, it is consistent (and likely as constructed) that $\betan$
will have symmetric tie-points of type $(\mathfrak c,\mathfrak c)$ in
the model $V[G\cap P]$ and $V[G]$. 
\end{remark}

\begin{remark} In the proof of Theorem \ref{main2}, it is easy to
  arrange that each $K_\lambda$ ($\lambda\in \Lambda$) is also
  $K_{\mathfrak T_\lambda}$ for a ($T_\lambda$-generic) full tower, 
$\mathfrak T_\lambda$, of
 ${\Naturals}$-involutions. However the generic sets added by the
 forcing $P$ will 
  prevent this tower of involutions from extending to a full
  involution.
\end{remark}

\printthequestions

\end{document}